\documentclass[12pt]{amsart}
\usepackage[margin=2.1 cm]{geometry}
\usepackage{amsmath}
\usepackage{amsfonts}
\usepackage{amssymb}
\usepackage{xcolor}
\renewenvironment{proof}{{\bfseries Proof.}}{\qed}
\numberwithin{equation}{section} 
\theoremstyle{theorem}
\newtheorem{theorem}{Theorem}[section] 
\newtheorem{pro}[theorem]{Proposition} 
\newtheorem{cor}[theorem]{Corollary} 
\newtheorem{lemma}[theorem]{Lemma} 
\theoremstyle{definition}
\newtheorem{defn}[theorem]{Definition} 
\newtheorem{rk}[theorem]{Remark} 
\newcommand{\Rp}{\mathbb{RP}^{2}}

\newcommand{\G}{$\Gamma$ }

\newcommand{\imp}{\Rightarrow}
\newcommand{\z}{\mathbb{Z}}

\newcommand{\R}{\mathbb{R}}

\newcommand{\s}{\sigma_{1}}
\newcommand{\si}{\sigma_{2}}

\newcommand{\gp}{$\Gamma^{\prime}$ }

\newcommand{\thmref}[1]{Theorem~\ref{#1}}
\newcommand{\lemref}[1]{Lemma~\ref{#1}}

\newcommand{\corref}[1]{Corollary~\ref{#1}}

\numberwithin{equation}{section} 
\def \h{{\bf H}} 
\newcommand{\bn}{B_{n}/\langle h\rangle}

\usepackage[backref]{hyperref}

\begin{document}
	\title{Reversible and other generalised torsion elements in Seifert-fibered spaces}
	\author{Anushree Das  \and   Debattam Das}
		\address{Indian Institute of Science Education and Research (IISER) Mohali,
		Knowledge City,  Sector 81, S.A.S. Nagar 140306, Punjab, India}
	\email{nshrdas@gmail.com}
	\email{debattam123@gmail.com}
	
	\maketitle
	\begin{abstract}
		An element $a$ in a group $\Gamma$ is called \emph{reversible} if there exists $g \in \Gamma$ such that $gag^{-1}=a^{-1}$. The reversible elements are also known as `real elements' or `reciprocal elements' in literature. In this paper, we classify the reversible elements in Fuchsian groups, and use this classification to find all reversible elements in a Seifert-fibered group. We then apply the classification to the braid groups, particularly to the braid group on $3$ strands. We further study generalised 3-torsion elements in PSL(2,Z), and use this to analyse the existence of generalised 3-torsion elements in Seifert-fibered spaces in general, and braid groups on 3 strands in particular.
	\end{abstract}
	\section{Introduction}
	A non-trivial element $a$ in a group $ \Gamma$ is said to be `reversible' if it is conjugate to its own inverse, and the conjugating elements are called `reverser' or `reversing symmetry'. If the conjugating element is also an involution, then it is called `strongly reversible' element. It has also been called by other terms like `reciprocals' or `real elements' in other works, e.g., \cite{BR}, \cite{La} etc. In this paper our scope is limited to some properties of reversible elements, and the existence of reversible elements in some discrete groups. Reversibility has its own geometric motivations and has been studied widely in literature. For instance, in a Fuchsian group, the collection of conjugacy classes of reversible elements have a 1-1 correspondence to the conjugacy classes of infinite dihedral groups contained in that Fuchsian group.
But in general, supposing $G$ is a 3-manifold group (the fundamental group of some 3-manifold), if $G$ is torsion free then a reversible element $g$ in $G$ is represented as a loop which is homotopic to its own inverse (i.e., it will form an embedded Klein bottle with the reversing symmetry). So, finding a reversible element with a reverser in a 3-manifold is equivalent to finding an embedded Klein bottle up to homotopy inside the manifold. In this paper we are specially interested in Seifert fibered manifolds since the fundamental group of a Seifert fibered manifold is a universal extension of Fuchsian groups or non-orientable 2-orbifold groups.

 Furthermore, an element $g$ in a group $G$ is said to be a generalised $n$-torsion element if it satisfies the condition 
\[gh_1gh_1^{-1}h_2gh_2^{-1}\dots h_{n-1}gh^{-1}_{n-1}=Id.\] Generalised $n$-torsion elements are conjugacy invariant. Reversible elements are therefore also known as generalised $2$-torsion elements ( see \cite{HMT}). We also study the existence of generalised $3$-torsion elements in $\rm PSL(2,\z)$ and use it to find such elements in Seifert-fibered spaces.

A Seifert fibered manifold is a compact $3$-manifold which can be seen as a $S^1$ bundle over a $2$-dimensional orbifold as the base space, such that every fiber has a solid fibered torus neighbourhood. These spaces were first studied by Herbert Seifert in his doctoral thesis \cite{ST}, and they form an important example of $3$-manifolds. In fact, all the the compact orientable $3$-manifolds in $6$ of the $8$ possible $3$-manifold geometric structures can be seen as Seifert fibered spaces. Seifert fibered spaces have been widely studied, for instance in \cite{JS}, \cite{LR}, \cite{KR}. If we consider the space of fibers of a Seifert fibered space $M$ with a quotient topology, we get a resulting surface called the orbit surface of $M$. We denote the orbit surface by $G$. The orbit surface can also be seen as a quotient of the base space of $M$. There is a map $\phi:\pi_1(G)\to \{-1,1\}$ called the classifying homomorphism. A loop $\alpha$ on $G$ has image $\phi(\alpha)=1$ if the fiber above $\alpha$ is mapped to itself in an orientation preserving way as we follow the loop $\alpha$, and has image $\phi(\alpha)=-1$ otherwise. The classifying homomorphism plays a crucial role in determining the fiber type of a Seifert fibered space.
	
	 Given any such Seifert fibered space, this paper is concerned with checking the reversibility of the elements of $\pi_1(M)$, the fundamental group of $M$. The fundamental group of $M$ depends on the orientability of the orbit surface $G$. If $G$ is orientable, it can be shown following Seifert's work and Brin's note in \cite{Br}, that the fundamental group of $M$ is of the form
	\begin{align}\label{fo}
	 \pi_1(M) = & \langle a_i,b_i,c_i,d_i,h~|~ a_iha_i^{-1}=h^{\phi(a_i)}, 
\end{align}
	 \[b_ihb_i^{-1}=h^{\phi(b_i)}, c_ihc_i^{-1}=h,d_ihd_i^{-1}=h^{\phi(d_i)},\] 
\[	c_i^{\mu_i}=h^\beta_{i},\prod~[a_i,b_i]~\prod c_i~\prod d_ih^b=1\rangle .\]
	
If $G$ is not orientable, $\pi_1(M)$ has the presentation 
	\begin{align}\label{fn}
	 \pi_1(M) =  \langle x_i,c_i,d_i,h~|~ x_ihx_i^{-1}=h^{\phi(x_i)}, \end{align}
	\[c_ihc_i^{-1}=h,d_ihd_i^{-1}=h^{\phi(d_i)}, 	c_i^{\mu_i}=h^{\beta_i},\prod x_i^2~\prod c_i\prod d_ih^b=1\rangle .\]  
	 In both cases the $\mu_i$, $\beta_i$ and $b$ are the same as used in the representation of $M$, and $h$ represents an ordinary fiber. The $a_i$ and $b_i$ are the generators of the handlebodies of the orbit surface, the $d_i$ represent the boundaries of the orbit surface, the $x_i$ generate the crosscaps in the orbit surface, and the $c_i$ are the generators of the essential fibers. $\phi$ is the classifying homomorphism. $\langle h\rangle$ is a normal cyclic subgroup of $\pi_{1}(M)$, and so we can consider the quotient group $\pi_1(M)\slash \langle h\rangle $. If the orbit surface $G$ is orientable, this quotient is a Fuchsian group which can be written as  	
	\begin{align*}	
	 \pi_1(M)\slash \langle h\rangle  = & \langle a_i,b_i,c_i,d_i|c_i^{\mu_i}=1,\prod[a_i,b_i]~\prod c_i~\prod d_i=1\rangle  . 
	\end{align*}\label{fo}
If $G$ is non-orientable, then 
\begin{align*}	
	 \pi_1(M)\slash \langle h\rangle  = & \langle x_i,c_i,d_i~|~c_i^{\mu_i}=1,\prod~ x_i^2\prod~ c_i\prod~ d_i=1\rangle .  
	\end{align*}\label{fn}
	From the representation of $\pi_1(M)$, we can see that the reversibility of the element $h$ is dependent on the classifying homomorphism $\phi$. If there is any generator $z$ of the form $a_i, b_i,x_i,$ or $d_i$ such that $\phi(z)=-1$, then the relations in the definition of $\pi_1(M)$ give that $h$ is reversible, i.e., $zhz^{-1}=h^{\phi(z)}=h^{-1}$. Our aim  is to classify all the reversible elements of $\pi_1(M)$.  The main result of our paper is:
	\begin{theorem}\label{l1}
		If a Seifert fibered space $M$ has an orientable orbit surface. If $\phi$ is nontrivial, the reversible elements of $\pi_1(M)$ are the following:
		\begin{enumerate}
			\item the elements conjugate to any power of $h$.
			\item  the conjugates of the elements of the form  $c_i^{\frac{\mu_i}{2}}kc_j^{\frac{\mu_j}{2}}k^{-1}$ with $\phi(k)=-1$  where $\mu_i$ and $\mu_j$ are even and $k$ is any element in $\pi_1(M)$, and $\beta_{i}=\beta_{j}$.
			\item  the conjugates of the elements of the form $c_i^{\frac{\mu_i}{2}}kc_j^{\frac{-\mu_j}{2}}k^{-1}$ with $\phi(k)=1$, where $\mu_i$ and $\mu_j$ are even and $k$ is any element in $\pi_1(M)$, and $\beta_{i}=\beta_{j}$.
		\end{enumerate}
		If $\phi$ is trivial,  $\pi_1(M)$ has reversible elements as the elements conjugate to $c_i^{\frac{\mu_i}{2}}kc_j^{\frac{-\mu_j}{2}}k^{-1}$.
	\end{theorem}
	
	We also similarly classify the reversible elements in case of Seifert fibered spaces with non-orientable orbit surfaces.

	In \textbf{Section 2}, we study the fundamental groups of two-dimensional orbifolds, and classify the reversible elements of such groups. This classification is then applied to the fundamental groups of any Seifert-fibered space in \textbf{Section 3}. \textbf{Section 4} then discusses an application of the results of \textbf{Section 3} on the fundamental group of $B_3$, the braid group on $3$ strands. The possibility of using the same methods for any $\bn$  is also discussed there. In \textbf{Section 5}, we study the existence of generalised $3$-torsion elements in Seifert-fibered spaces, and apply it in $\pi_1(B_3)$ to find such elements. We also find conditions for the existence of generalised $n$-torsion elements.
	
	 After completion of the first draft, the authors were informed by Prof. Masakazu Teragaito regarding the article \cite{HMT} written by Himeno, Motegi, and Teragaito, where they classify 3-manifolds whose fundamental groups admit generalised torsion elements of order two, and further classify such elements. Their proof is topological and completely different from the proof in the \textbf{Section 3} of this paper. The two articles were written completely independently and at around the same time. The proof in \cite{HMT} crucially uses important results by Jaco-Shalen and Hass, while the techniques in this article are simpler in comparison. 
	\subsection*{Acknowledgment}
The authors would like to acknowledge and thank Prof. Masakazu Teragaito for his helpful comments and for pointing out an error in the first draft. The authors further thank Prof. Krishnendu Gongopadhyay and Dr. Soma Maity for helpful feedback. A. Das acknowledges full support from an NBHM research fellowship during this work. D. Das acknowledges full support from a UGC-NFSC research fellowship during this work. 
	\section{Reversiblity in Fundamental groups of the surfaces}
	
	\subsection{Reversible classes of Fuchsian group}
		We can see the classifications of reciprocal elements for any finite type Fuchsian groups. We have discussed the reciprocal elements of Hecke groups in \cite{DG}, which is a particular class of Fuchsian groups.
	
	From now on, the hyperbolic axis for an hyperbolic element $ h $ is written as $ I_{h} $.
	\begin{theorem}\label{thm3}
		Let $\Gamma$ be an arbitrary finitely generated Fuchsian group. Let $g \in \Gamma$. Then the following are equivalent.
		\begin{enumerate}
			\item $g$ is reciprocal. 
			\item  $g$ is either an involution or conjugate to a hyperbolic element $h$ such that the hyperbolic axis of $h$ passes through some of the even order vertices of the fundamental domain of the group \G.
		\end{enumerate}  
	\end{theorem}
	To prove this theorem, we first need to prove the following lemmas.
	\begin{lemma}\label{lem5}
		If a hyperbolic element $ h $ in the Fuchsian group $ \Gamma $ is reciprocal, then the fixed point of the reciprocator of the reciprocal elements lies on the hyperbolic axis of $ h $, i.e., on $ I_{h} $.
	\end{lemma}
	\begin{proof}
		$\Gamma$ is in $ \rm PSL(2,\R) $, and so we will conjugate \G by a M\"{o}bius map $f $ such that in the new group after conjugation, $ h $ will be a hyperbolic element in the Poincar\'{e} disk where the hyperbolic axis of $ h $ passes through $ (0,0) $. Now, the reciprocator for $ h $ is $ g $, which is an involution having fixed point $ z $. So, let us suppose $ z $ does not lie on $ I_{h} $ and $ I_{h} $ is a straight line passing through the origin. Then, $ z $ lies one side of the axis. In our earlier work {\cite[Lemma 2.2]{DG} }we have shown that in the Fuchsian group, the reciprocators of a reciprocal element fix the hyperbolic axis of that particular hyperbolic element and interchange the fixed points of hyperbolic axis. Now, we apply the map $ g $ to the disk. This map is a $\pi$-rotation along $ I_{h} $ of the disk, and therefore the point $ z $ must change its position. This is contradiction. So, the fixed point of $ g $ must lie on $ I_{h} $. 
	\end{proof}
	\begin{lemma}\label{lem6}
		A hyperbolic element $ g $ in the Fuchsian group \G is  reciprocal if and only if it passes through some of the even-order vertices of the tessellation which is done by the fundamental domain of \G.
	\end{lemma}
	\begin{proof}
		We first show that if $ g $ is reciprocal, then $ I_{g} $ passes through at least one even order vertex in the fundamental domain. First, the fixed points of the reciprocator of $ g $ will always be a vertex of the domain. We know that $ I_{g} $ passing through a fixed domain represents a $ z $-class of that Fuchsian group. {And if it is reciprocal then there is a fixed point of one of the reciprocator of $ g $ that lies on $ I_{g} $ by \ref{lem5}, and it is an even ordered-vertex of the fundamental domain (from the definition of fundamental domain)}. This proves one side of the statement.

		Conversely, if the axis of the primitive hyperbolic element $ g $ passes through an even-order vertex $ z $ of the tiling which is done by the fundamental domain, then let $ k $ be the elliptic finite order element in the Fuchsian group such that, $ k(z)=z $ and $ k^{2m}=Id $. {So, $ k^{m} $ is an involution and it is a is $ \pi $-rotation. Then, it will interchange the fixed points of $ g $, and so $ k^{m}gk^{m}=g^{-1} $ since $ g $ is primitive.} This shows that any non-trivial power of $ g $ is reciprocal. This proves the lemma.
		
	\end{proof}
	\subsection{Proof of \thmref{thm3}}
	
	Combining \lemref{lem5} and \lemref{lem6}, we get the proof of the theorem.
\qed

	\subsection{Reversible classes of non-orientable surface fundamental groups}
	Let $S$ be a compact connected non-orientable surface. Due to the classification of surfaces, the fundamental group of $S$ can be presented in the form
\begin{align}\label{nf}
		\langle x_{1},x_{2},\dots ,x_{k},d_{1},d_{2},\dots,d_{m}|\Pi x^{2}_{i}\Pi d_{i}=1\rangle
\end{align}
	where $k$ represents number of cross-caps and the $d_{i}$ represent the boundary components. In such cases, we have the following lemma.
	
	\begin{lemma}
	If	$S$ is a compact, connected non-orientable surface having the fundamental group \ref{nf}, then there are no reversible elements in $\pi_{1}(S)$ except when $S$ is the Klein-bottle or $ \Rp $.
	\end{lemma}
\begin{proof}
By the classification of non-orientable surfaces, we know that $ S $ is  homeomorphic to $ \underbrace{\Rp\#\Rp\#\dots \# \Rp} _{k\ times} $. We also know that a loop $ l $ is reversible iff it will freely homotope to its own inverse, that is, they will form a Klein bottle. So, if $ S $ is itself a Klein bottle, then any power of the generating loop is reversible. If $ S $ is $ \Rp $, then the fundamental group will be $ \z_{2} $, and the generator is reversible.

If $ k > 2 $, there is no finite order element in the fundamental group. Then our claim is that there are no more reversible elements in $ S $. If on the contrary there exists any loop $ l $ that homotopes to its own inverse, then $\pi_{1}(S)$ will no longer be a biorderable group. This will lead us to a contradiction. Hence, this proves the lemma.


\end{proof}\\
Now we classify the reversible elements in the fundamental group of a non-orientable two-dimensional orbifold $\Gamma^{\prime}$ {which is neither $\Rp$ nor Klein bottle}. 
	\begin{lemma}\label{main}
	Infinite order reversible elements in \gp are strongly reversible elements, and every reverser is an involution. 
	\end{lemma} 
	The proof is similar to the proof {of} \cite[Lemma 3.4 ]{DG}.\\
	\begin{proof}
		The group \gp is a discrete subgroup of $\rm PGL(2,\R)$. Let $r$ in $\h^{2}/\Gamma^{\prime}$ be a reversible element in \gp. Then there is an element $b$ in \gp such that $b^{-1}rb=r^{-1}$. At first we consider that $r$ is orientation preserving isometry in $\h^{2}$ (which implies that will not be a parabolic element). Then, $b$ either fixes the fixed points of $r$ or reflects the fixed points to each other. If $b$ fixes the set of fixed points of $r $ then it will either commute with $r$ or reverse $ r$. In the case where it commutes with $r$ {it leads us to a contradiction that $b$ is not a reverser}, and in the case where it is reversing, $b$ will be an involution. If $b$ interchanges the fixed points, then $b$ must be an involution. So, every reverser is an involution in \gp.
		
		In  $r$ is a orientation reversing element, then $r^{2}$ will be orientation preserving element with the same reverser $b$. From the previous case, $b^{2}=1$. Hence, in any case, $r$ will be strongly reversible. 
	\end{proof}
	
	\begin{theorem}
		The infinite order reversible elements in \gp are conjugate to $c_i^{\frac{\mu_i}{2}}kc_j^{\frac{\mu_j}{2}}k^{-1}$ where $\mu_{i}$ and $\mu_{j}$ are even and $k\in \Gamma^{\prime}$.
	\end{theorem}
	\begin{proof}
		Since there is no finite order element in \gp other than the elements conjugate to $c_{i}$, this  theorem  follows from the previous lemma.
	\end{proof}

	\section{Seifert fibered space}  
	To find the reversible elements of the fundamental group of a Seifert fibered space, we first focus on the quotient group $\pi_1(M)\slash \langle h \rangle$ and discuss its reversibility. We first deal with those spaces which have orientable orbit surfaces.
\begin{cor}
		Any element conjugate to either $c_i^{\frac{\mu_i}{2}}$ or $c_i^{\frac{\mu_i}{2}}kc_j^{\frac{\mu_j}{2}}k^{-1}$ where $k\ \in \pi_1(M)\slash \langle h\rangle $ is reversible in $\pi_1(M)\slash \langle h \rangle$ iff $\mu_i$ and $\mu_j$ are even. No other elements in $\pi_1(M)\slash \langle h \rangle$ are reversible.
\end{cor}
	\begin{proof}
	Since $\pi_1(M)\slash \langle h\rangle $ is a Fuchsian group $\Gamma$, all reversible elements in $\Gamma$ are strongly reversible. By {\cite[Theorem 1.1]{DG}}, it is known that an element in a Fuchsian group is reversible if and only if it is an involution or a product of $2$ involutions. By the representation of the group, the involutions are the elements conjugate to any $c_i^{\frac{\mu_i}{2}}$ where $\mu_i$ is even. Hence, the claim follows.
	\end{proof}\\
	
	
	Now we can discuss the reversibility of the fundamental group of $ M $. 
	Since there is no finite order element in $ \pi_{1}(M) $,  we cannot talk about strongly reversible elements here. But there may be reversible elements, as stated earlier. Not that if an element is reversible, then so are its conjugates.
	
	\begin{lemma}
Elements conjugate to the power of the fiber generator $ h $ is a reversible element iff the classifying homomorphism on $\pi_1(G)$ is non-trivial.
	\end{lemma}

	\begin{proof}
Since $\phi$ is a homomorphism, it is determined by its actions on the generators of $\pi_1 (G)$. If $\phi$ is non-trivial, there is some generator $x$ of the form $a_i$, $b_i$, $x_i$ or $d_i$ such that $\phi(x)=-1$. Then by the presentation of $\pi_1 (M)$, $xhx^{-1}=h^{\phi(x)}=h^{-1}$, and so $h$ is reversible.
\end{proof}

\begin{cor}
	If $h$ is reversible in $\pi_1(M)$, the $c^{\mu_{i}}_i$ are reversible elements although each $ c_{i} $ is an infinite order element in $ \pi_{1}(M) $.
\end{cor}
\begin{proof}
	Let $h$ be reversible. Then so is $h^{\beta_i}$. The relation $c_i^{\mu_i}=h^{\beta_i}$ now implies that $c_i^{\mu_i}$ is also reversible.
\end{proof}

The two following lemmas investigate the reversibility of the other elements of $\pi_1(M)$.
 
\begin{lemma}\label{3.4}
	The powers of $c^{\mu_i}_{i}$ are reversible, but any other multiple of $c_{i}$ is not reversible for any $\pi_{1}(M)$.
\end{lemma}
\begin{proof}
	On the contrary, let a multiple of $c_{i}$ be a reversible element in $\pi_1(M)$ with reverser $\rho $. We can consider that $\rho$ is not a power of $h$, since $h$ is contained in the centralizer of $c_{i}$. Then, the reverser $\rho$ is transverse to $h$, and this implies any power of $\rho $ will always be transverse to $h$. Hence $\rho$ has a non-trivial image in the group $\pi_{1}(M)/\langle h\rangle$, denoted by $[\rho]$. Since, from \cite{DG}, every reversible element in a Fuchsian group is strongly reversible, $[\rho]$ must be of order $2$. This implies, $\rho^{2}=h^{k}$ for some $k\in \z$, which contradicts our assumption. So, there is no such reverser that reverses any other power of $c_{i}$ except for the powers of the $c^{\mu_{i}}_{i}$.
\end{proof}

{\begin{lemma}\label{label}
If $M$ has an orientable orbit surface, the only reversible elements in $\pi_1(M)$ can be the conjugates of the powers of $h$ and the conjugates of the elements either of the form  $c_i^{\frac{\mu_i}{2}}kc_j^{\frac{\mu_j}{2}}k^{-1}$ with $\phi(k)=-1$ or of the form $c_i^{\frac{\mu_i}{2}}kc_j^{\frac{-\mu_j}{2}}k^{-1}$ with $\phi(k)=1$, where $\mu_i$ and $\mu_j$ are even, and $k$ is any element in $\pi_1(M)$, $\beta_{i}=\beta_{j}$.
\end{lemma}}

	\begin{proof}
	Let $x$ be an element in $\pi_{1}(M)$, other than any power of $h$.
	 Let $x$ be reversible. Then $x$ is freely homotopic to $x^{-1}$. Consider the projection of this free homotopy to the Fuchsian group $\pi_1(M)/\langle h \rangle$. This would be a free homotopy of a loop in this Fuchsian complex to its inverse. But this would imply that this is a reversible element in $\pi_1(G)/\langle h \rangle $. But as we have shown earlier, the reversible elements in this group are of the form of conjugates of $c_i^{\frac{\mu_i}{2}}kc_j^{\frac{\mu_j}{2}}k^{-1}$ or $c_i^{\frac{\mu_i}{2}}$ for even $\mu_i$ and $\mu_j$. But, for elements that are conjugates of $c_i^{\frac{\mu_i}{2}}$, i.e., $\alpha c^{\frac{\mu_{i}}{2}}\alpha^{-1}$, the reverser of $x$ must be $x$ itself. Then the element in $x\in \pi_{1}(M)$ will satisfy,\[ \alpha c^{\frac{\mu_{i}}{2}}\alpha^{-1}\alpha c^{\frac{\mu_{i}}{2}}\alpha^{-1}\alpha c^{\frac{-\mu_{i}}{2}}\alpha^{-1}=\alpha c^{\frac{-\mu_{i}}{2}}\alpha^{-1} \]
	 \[\imp h^{\beta_{i}}=1.\]
	 Since $h$ is an infinite order element, this is a contradiction. Hence, the lifts of the elements conjugate to $c^{\frac{\mu_{i}}{2}}$ are not reversible in $\pi_{1}(M)$. So, the only other candidates from $\pi_{1}(M)/\langle h\rangle$ that may be reversible in $\pi_{1}(M).$ are the products of the involutions.\\
	 Let the element $x=c_i^{\frac{\mu_i}{2}}kc_j^{\frac{\mu_j}{2}}k^{-1}$ be reversible in $\pi_{1}(M)/\langle h\rangle$ with reverser $c^{\frac{\mu_{i}}{2}}.$ Then in $\pi_{1}(M),$
	 \[ c_{i}^{\frac{\mu_{i}}{2}} c_i^{\frac{\mu_i}{2}}kc_j^{\frac{\mu_j}{2}}k^{-1}c_{i}^{-\frac{\mu_{i}}{2}} c_i^{\frac{\mu_i}{2}}kc_j^{\frac{\mu_j}{2}}k^{-1}= c_{i}^{\mu_{i}} kc_j^{\frac{\mu_j}{2}}k^{-1}kc_j^{\frac{\mu_j}{2}}k^{-1}= h^{\beta_{i}}kh^{\beta_{j}}k^{-1}. \]
	 Then, $h^{\beta_{i}}kh^{\beta_{j}}k^{-1}$ will be identity only if $\beta_{i}=\beta_{j}$ and $ \phi(k)=-1.$ 
	 It will be same if we choose any other reverser for that particular reversible element since all the reversers are in the form of $x^{l}c^{\frac{\mu_{i}}{2}}.$ In the similar way, we can show that $c_i^{\frac{\mu_i}{2}}kc_j^{-\frac{\mu_j}{2}}k^{-1}$  is reversible if $\phi(k)=1$.
	
	If $x$ is not reversible in $\pi_1(M)/\langle h \rangle$, a reverser of $x$ in $\pi_1(M)$, say $\rho$, must be a power of $ h$, since the projection to $\pi_1(M)/\langle h \rangle$ must be constant. Then in $\pi_{1}(M)/\langle h\rangle  $, $[x]=[x^{-1}]$. This implies that $[x^{2}]=1$. Then, $x^{2}$ is a power of $h$, and $x^{2}$ is also reversible with same reverser $\rho$ that $x$ is. Then, \[ h^kx^{2}h^{-k}=x^{-2}\]
		\[\imp x^{2}=x^{-2}\]
		\[\imp x^{4}=Id.\] This is contradiction since $\pi_{1}(M)$ is torsion free. This proves our lemma.
	\end{proof}
\subsection{Proof of \thmref{l1}}
The previous lemma shows that there are no reversible elements in $\pi_1(M)$ other than those conjugate to $h$. As shown earlier, the reversibility of $h$ and its conjugates depends on the classifying homomorphism. This is summarised in the statement of Theorem \ref{l1}.
\qed

In the case of non-orientable surfaces, we have certain exceptions to this situation, as shown in the previous section.

The reversible elements of $\pi_1(M)$ in case of a non-orientable orbit surface can be summarised as follows, by the same proofs as above.

\begin{cor}
If $M$ has an non-orientable orbit surface other than the Klein bottle or $\Rp$, the reversible elements of $\pi_1(M)$ are either conjugate to any power of $h$ or either of the form  $c_i^{\frac{\mu_i}{2}}kc_j^{\frac{\mu_j}{2}}k^{-1}$ with $\phi(k)=-1$ or of the form $c_i^{\frac{\mu_i}{2}}kc_j^{\frac{-\mu_j}{2}}k^{-1}$ with $\phi(k)=1$, where $\mu_i$ and $\mu_j$ are even, and $k$ is any element in $\pi_1(M)$, $\beta_{i}=\beta_{j}$.
\end{cor}

When the orbit surface is the Klein bottle or $\Rp$, we get additional reversible elements which are the generators of the orbit surface. Thus, in the case of a given orbit surface, the reversibility of the elements of the fundamental group of a Seifert fibered space over that orbit surface is solely determined by $\phi$. 

\begin{rk}
	The classification of the reversible elements in a given Seifert-fibered space depends only on the classifying homomorphism.
\end{rk}

\section{Application to universal extension of Fuchsian groups}
In the previous section, we have classified the reversible elements of the fundamental group of a Seifert-fibered  space. We can see the fundamental group in the form of an exact sequence:
\begin{equation}
		1\longrightarrow \z \longrightarrow \pi_{1}(M)\longrightarrow F\longrightarrow 1
\end{equation}

Here, $\pi_{1}(M)$ is the fundamental group of $M$, and $F$ is a Fuchsian group or  discrete isometry group of $\rm PGL(2,\R)$. So, $\pi_{1}(M)$ is then a universal extension of $F$, which can now be seen as $\pi_1(M)\slash \langle h \rangle$. This gives us a new view-point to consider certain groups as fundamental groups of a Seifert-fibered space. Let us consider $F=\rm PSL(2,\z)$ which is presented by 
	$\langle a,b| a^{2},b^{3}\rangle.$ We have already mentioned the presentation of $\pi_{1}(M)$ for a given $M$ in Equation \eqref{fo} and Equation \eqref{fn}.
	
 We can choose an orientable connected $3$-manifold $M$ in such a way that the $\pi_{1}(M)\cong B_{3}$. $B_{3}$ has a presentation as $\langle \s \si | \s \si \s=\si\s\si\rangle $. The normal subgroup of $B_{3}$ is $\z$ and it is generated by $(\s\si)^{3}$. Thus, $B_{3}/\langle \s\si \rangle=\langle \s\si\s, \s\si |(\s\si)^{3},(\s\si\s)^{2} \rangle=\rm PSL(2,\z)\cong\Delta(2,3,\infty)$. Choosing the $c_{i}$ and $\beta_i$ suitably, we can consider the Seifert-fibered manifold $M$ determined by $ (O,o,0|1, (2,1),(3,1))$. The presentation of the fundamental group of $M$ is then
 \begin{equation}\label{eq}
 		\langle c_{1},c_{2},d,h| c_ihc_i^{-1}=h,dhd^{-1}=h^{\phi(d)} ,	c_i^{\mu_i}=h^\beta_{i}, c_1 c_{2} d h=1 \rangle.
 \end{equation}

 Now, consider the generating elements $c_{1}=\s\si\s,c_{2}=\s\si,$ which implies that $h=(\s\si)^{3}$ and $d^{-1}=\s\si\s^{2}\si.$\\ This determines the Fuchsian complex $\pi_1(M)\slash \langle h \rangle$, and we obtain the resultant orbit surface $G$ as the open disk $D$ by further quotienting it by the $c_i$. Then, $\phi : \pi_{1}(D)\longrightarrow \z_{2}$ is the trivial homomorpshim. Then,  the group \ref{eq} can be presented as
 \begin{align*}
 	\langle \s\si\s,\s\si,\s\si\s^{2}\si,(\s\si)^{3}|(\s\si)^{3}\s\si\s=\s\si\s(\s\si)^{3}
 	,(\s\si\s)^{2}=(\s\si)^{3}
 	 \rangle
 \end{align*}
 \begin{align*}
 	=\langle \s,\si| \s\si\s=\si\s\si\rangle=B_{3}.
 \end{align*}
 So, we can deduce the exact sequence 
 \begin{align}
 	1\longrightarrow \z \longrightarrow B_{3}\longrightarrow \rm PSL(2,\z)\longrightarrow 1
 \end{align}
\begin{cor}\label{4.1}
	The only reversible elements of $B_{3}$ are conjugates to $[\s\si\s,k]$ for any element $k$ in $B_{3}$. 
\end{cor}
\begin{proof}
	From the above discussion, $\phi$ is trivial, and  applying Theorem \ref{label} we know that reversible elements  for the orientable case depend only on $\phi$. The claim thus follows from this theorem.
\end{proof}

Note that we cannot use the same technique to find the reversible elements of $B_n$ since $B_{n}$ is not a Seifert-fibered group for any $n>3$. This follows from the theorem below.

\begin{theorem}\label{4.2}
	For any $n>3$ the braid group $B_{n}$ is not isomorphic to any Seifert-fibered group. 
\end{theorem}
\begin{proof}
	On the contrary let us assume that braid groups are Seifert-fibered groups. So, let there be a Seifert-fibered manifold $M$ such that $B_{n}\cong \pi_{1}(M)$, with maximal normal subgroup $\langle h\rangle$ where $h$ represents an ordinary fiber. It is well known that the braid group $B_{n}$ where $n>2$ is not virtually abelian. Then, from \cite[Corollary 10.4.6]{Ma}, the center of $\pi_{1}(M)$ is either $\langle h\rangle$ or trivial. Since the center of $B_{n}$ is not trivial, therefore $Z(B_{n})$ is $\langle h\rangle$. Let us consider the braid group $B_n$
	\[B_{n}=\langle \s\si\dots \sigma_{n-1}|\sigma_{i}\sigma_{{i+1}}\sigma_{i}=\sigma_{{i+1}}\sigma_{i}\sigma_{{i+1}}, \sigma_{i}\sigma_{j}=\sigma_{j}\sigma_{i} \hbox{  for  } |i-j|>2\rangle.\] 
	In another presentation we can rewrite the group as mentioned in \cite[Chapter 2, Section 2, Exercise 2.4]{MK}:
	\begin{equation}\label{4.4}
		B_{n}=\langle x,y|x^{n}=(xy)^{n-1}, x^{i}yx^{-i}y=yx^{i}yx^{-i} \hbox{ where } i=2,3,\dots, m\rangle
	\end{equation}
	where, $m=\lfloor\frac{n}{2}\rfloor.$ 
	Then, the fiber $h=(\s\si\dots\sigma_{n-1})^{n}=x^{n}$ is the generator of $Z(\pi_{1}(M))$. 
	 Then, \[B_{n}/\langle h\rangle=\langle  x,y|x^{n}=(xy)^{n-1}=1, x^{i}yx^{-i}y=yx^{i}yx^{-i} \hbox{ where } i=2,3,\dots, m\rangle \] is a Fuchsian group or a fundamental group of non-orientable orbifold surface. So, it acts as an isometry group in $\h^{2}$. If $\bn$ acts in $\h^{2}$, then $x^{i} yx^{-i}$ and $y$ are powers of the same element $\psi\in \bn$, that is, $x^{i}yx^{-i}=\psi^{u}$ and $ y=\psi^{v}$ for some integers $u,  v$. But in \cite{DG}, we noticed that $x^{i}$ sends  fixed points of $y$ to the fixed points of $x^{i}yx^{-i}$. That means either $x^{i}$ fixes the fixed points of $y$, or interchanges them for every $i$. {If it fixes the fixed points then that implies $xy=yx.$ Then $x$ is an element in the center of $\bn $, so it must be trivial. This is contradiction in our assumption.}
	 
	 {So, $x^{i}$ must interchange the fixed points of $y$ for every $i$.} Then $x^{2}$ must be an involution (orientation preserving or orientation reversing). This implies $n=4$, and thus $x^{2}yx^{-2}=y^{-1}.$ {$y^{-1}$ is $\sigma_{3}$ in $B_{4},$ and then $\s \sigma_{3}$ is an element in the center which contradicts our assumption.}
	  So, $\bn$ will not be the fundamental group of a Seifert-fibered space for any $n>3$.
\end{proof}\\
However, there always exits some reversible elements in $B_{n}$ by the following remark. Here, we are assuming the presentation of $B_{n}$ as given in \eqref{4.4}.
\begin{rk}
	From \corref{4.1}, there are some reversible elements in $B_n$ which are conjugate to the elements of the form $[a,b]$, where the image of $a$ is an involution in the quotient space $B_{n}/\langle x^{n}\rangle $ and $b$ is any other element in $B_{n}$. This is a sufficient condition to be a reversible element in $B_{n}$.
\end{rk}
\begin{rk}
	From the above theorem, we obtain that the same idea that we have used for Seifert-fibered spaces will not work for $B_{n}$. Furthermore, by \cite[Theorem 1.3]{HMT} and \thmref{4.2}, $B_n$ cannot be a fundamental group of any $3$-manifold. This has been proven earlier using different techniques, for example in \cite{Go} where they use the theory of coherent groups, however this gives a different perspective using reversible elements.
\end{rk}

\section{ Generalised 3-Torsion elements in $\pi_{1}(B_{3})$ }
\begin{defn}
	An element $g$ in the group $G$ is said to be a generalised $n$-torsion element if it satisfies the following condition:
	\[k_{1}gk_{1}^{-1}k_{2}gk_{2}^{-1}\dots k_{n}gk_{n}^{-1}=Id\] where, $k_{1},k_{2},\dots,k_{n}\in G.$
\end{defn}
Note that, a generalised $2$-torsion element is equivalent to a reversible element by definition. This property is invariant over conjugacy and inversion. In this section, we study the generalised $3$-torsion elements in the braid group with $3$-strains $B_{3}$, and provide a sufficient condition for an element to be a $3$-torsion element. To do that, we first check the sufficient conditions for the generalised $3$-torsion elements in $\rm PSL(2,\z)$.

\begin{lemma}\label{5.2}
	If $g\in \rm PSL(2,\z) $ is a generalised $3$-torsion element then, $g$ is a product of $3$-torsion elements. 
\end{lemma}
\begin{proof}
    Let us consider the presentation of $\rm{ PSL(2,\z)}=$ $\langle a,b~|~a^{2},b^{3}\rangle$.
	In $\rm PSL(2,\z)$, let $g$ be a generalised $3$-torsion element. If we abelianise $\rm PSL(2,\z)$, then $[g]^{3}=Id$. That means, $g$ is a $3$-ordered element or it is a product of $3$-ordered elements in $\rm PSL(2,\mathbb{Z})$. The commutator $[a,b]$ of $a,b$ is also a product of two 3-ordered elements. Then, the reduced form of $g$ is $\Pi_{i=1}^{m}k_{i}[a,b]k_{i}^{-1}\Pi_{j=1}^{n} b_{i}$ where, $b_i,k_{i}\in \rm PSL(2,\z)$ and $|b_{j}|=3$. 
\end{proof}

If $g$ is an elliptic element in $\rm PSL(2,\mathbb{Z})$, then it is known to be of order either $2$ or $3$. Thus, by the previous lemma, a generalised $3$-torsion element which is also an elliptic element must be of order $3$, since the product of three ordered elements cannot be an involution.

Therefore, we can now assume that $g\in \rm PSL(2,\z)$ is a non-elliptic generalised $3$-torsion element i.e., there are $h,k\in \rm PSL(2,\z)$ such that 
\[ghgh^{-1}kgk^{-1} =Id\]and also,  $g$ is conjugate to $ab^{k_{1}}ab^{k_{2}}\dots ab^{k_{n}}$ where, $k_{i}=\pm 1$. Since the property of generalised $3$-torsion element is invariant over conjugacy, it is sufficient to consider\begin{equation}\label{5.1}
	g=ab^{k_{1}}ab^{k_{2}}\dots ab^{k_{n}}
\end{equation} where, $k_{i}=\pm 1$ and $1\leq k_{i}\leq n$.

\begin{lemma}\label{5.4}
	In $\rm PSL(2,\mathbb{Z})$, there are only two parabolic elements which are generalised $3$-torsion elements upto conjugacy, $(ab)^{\pm2}$.
\end{lemma}
\begin{proof}
	There is only one type of parabolic conjugacy classes in $\rm PSL(2,\mathbb{Z)}$ i.e., $(ab)^{n}$, where, $n\in \z$.
	The matrix form of $(ab)^{n}$ is $\begin{bmatrix}
		1 &&n\\
		0&&1
	\end{bmatrix}$. Let $g=(ab)^{n}$ be a generalised $3$-torsion element i.e., there exists $h,k \in \rm PSL(2,\z)$ such that,
	\[ghgh^{-1}kgk^{-1}=Id\]
	\[\imp ghgh^{-1}=kg^{-1}k^{-1}.\]
	Thus, $ghgh^{-1}$ is a parabolic element, since any conjugate of $g$ or $g^{-1}$ needs to be parabolic. Therefore,
	$|trace(ghgh^{-1})|$ is 2. Suppose $h$ has the matrix form $ \begin{bmatrix}
		a&&b\\
		c&&d
	\end{bmatrix}.$ Let the matrix form of $kg^{-1}k^{-1}$ be denoted by $A$. Then,
	\[A=\begin{bmatrix}
		1&&n\\
		0&&1
	\end{bmatrix}\begin{bmatrix}
a&&b\\
c&&d
	\end{bmatrix}\begin{bmatrix}
	1&&n\\
	0&&1
	\end{bmatrix}\begin{bmatrix}
	a&&b\\
	c&&d
	\end{bmatrix}^{-1}
	\]
	\[=\begin{bmatrix}
		1&&n\\
		0&&1
	\end{bmatrix}\begin{bmatrix}
		a&&b\\
		c&&d
	\end{bmatrix}\begin{bmatrix}
		1&&n\\
		0&&1
	\end{bmatrix}\begin{bmatrix}
	d&&-b\\
		-c&&a
	\end{bmatrix}\]
	\[=\begin{bmatrix}
		a+nc&& b+nd\\
		c&&d
	\end{bmatrix}\begin{bmatrix}
	d-cn&&-b+an\\
	-c&&a
	\end{bmatrix} 
	\]
	\[=\begin{bmatrix}
		1-n^{2}c^{2}-acn&&acn^{2}+a^{2}n+n\\
		-c^{2}n&&1+acn	\end{bmatrix}.\]
	Since, $A$ is parabolic, this implies, $|trace(A)|=2.$ The obvious way to get $|trace(A)|=2$ is if we consider $c=0$, which will give us $h=g^{i}$ for some $i\in \z$. Then $A$ will be $g^{2}$, which cannot be true, since $g^{-1}$ is not conjugate to $g^2$.
	Now, we consider $c\neq 0.$ That means 
	
	\[|trace(A)|=2\]
	\[\imp|2-n^{2}c|=2\]
	\[\imp 2-n^{2}c=-2\]
	since the $ 2-n^{2}c=2$ case will be reduced to $c=0.$ So,
	\[ n^{2}c=4.\]
 That implies, either $ n=1,\ c=4 ~\hbox{or}~n^{2}=4,\ c=1$. Since $g$ cannot be seen as a product of $3$-torsion elements if we take $n=1$, therefore we conclude that $g$ may be one of $(ab)^{\pm 2}.$

Without loss of generality, let us choose $g$ to be $(ab)^{2}.$ This implies, $g=abab.$
Also note that, $ababb^{-1}abab^{-1}baba=Id.$ Therefore, $ghgh^{-1}kgk^{-1}=Id$ where, $h=b^{-1}$ and $k=b.$ This verifies that $g$ is indeed a generalised $3$-torsion element.\end{proof}
\\
In the proof of the previous lemma, we have obtained that a parabolic element is the product of two 3-ordered elements. Therefore, let us now check if this is also true for hyperbolic elements.
\begin{lemma}\label{5.5}
	Let $g$ be a hyperbolic element in $\rm PSL(2,\z)$ and $g$ can be seen as a product of two 3-ordered elements i.e., $g=b_{1}b_{2}$, where $|b_{1}|=|b_{2}|=3.$ Then, $g$ is a generalised 3-ordered torsion element.
	
\end{lemma}
\begin{proof}
	$g=b_{1}b_{2}$ this implies, 
	\[b_{1}b_{2}b^{-1}_{2}b_{1}b^{-1}_{2}b_{2}b_{1}=Id\]
	\[\imp gb^{-1}_{2}gb_{2} b_{2}gb^{-1}_{2}=Id.\]
	Thate means,  $ghgh^{-1}kgk^{-1}=Id$ where, $h=b_{2}^{-1}$ and $k=b_{2}.$ Hence, this proves the lemma.
\end{proof}

\begin{theorem}
    Suppose $g$ be an element in $\rm PSL(2,\z)$, it is a generalized $3$-torsion element, if any of the followings hold:
	\begin{enumerate}
		\item  $g$ is an elliptic element, then $g$ must be order of $3.$
		\item  $g$ is a parabolic element, then $g$ is conjugate to any of $(ab)^{\pm 2}.$
		\item $g$ is a hyperbolic element, then $g$ is conjugate to an element which is product of two $3$-ordered elements. 
	\end{enumerate}
\end{theorem}
\begin{proof}
	$(1)$ is easy to see from the above discussion.
	$(2)$ follows from \lemref{5.4} and $(3)$ follows from \lemref{5.5}.
\end{proof}

Notice that the condition of the hyperbolic elements obtained in the previous theorem holds true for any other Fuchsian group. This motivates us to see the generalised $3$-torsion elements in the universal central extension of Fuchsian groups, i.e., in Seifert fibered groups.

\begin{pro}
 	Let us consider a Seifert-fibered space $M$ with fundamental group \eqref{fo} with some elements $c_{1}$ and $c_{2}$ such that $c_{1}^{3p}=h^{\mu_{1}}$ and $c_{2}^{3q}=h^{\mu_{2}}$ for some $p,q\in \mathbb{N}$. Then $e_{1}^{p}e_{2}^{q}h^{x}$ is a generalised $3$-torsion element if $3x+\mu_{1}+\mu_{2}=0$, where $e_{i}$ are conjugate to $c_{i}$.
 \end{pro}

 \begin{proof}
 Consider two elements $c_1$ and $c_2$ in $M$ where $c_{1}^{3p}=h^{\mu_{1}}$ and $c_{2}^{3q}=h^{\mu_{2}}$ for some $p,q\in \mathbb{N}$ and $\mu_1, \mu_2 \in \mathbb{Z}$. We know that in the Fuchsian group $\pi_1({M})/\langle h\rangle$, the product of two $3$-ordered elements are generalised $3$-torsion elements. 
  
 Now we take one of the lifts of $e_{1}^{p}e_{2}^{q}$ in $\pi_{1}(M)$, where $e_{i}$ conjugate to $c_{i}$ for $i=1,2$, and let $g=e_{1}^{p}e_{2}^{q}h^{x}$, where $x\in \z.$ Consider $h_{1}=e_{2}^{-q}$ and $k=e_{2}^{q}$. If $g$ satisfies the condition \[ gh_{1}gh_{1}^{-1}kgk^{-1}=Id,\] 
 then \[ e_{1}^{p}e_{2}^{q}h^{x}e_{2}^{-q}e_{1}^{p}e_{2}^{q}h^{x}e_{2}^{q}e_{2}^{q}e_{1}^{p}e_{2}^{q}h^{x}e_{2}^{-q}=Id\]
 \[\imp e_{1}^{p}h^{x}e_{1}^{p}h^{x}e_{2}^{3q}e_{1}^{p}h^{x}=Id\]
 \[\imp e_{1}^{3p} h^{3x+\mu_{2}}=Id.\]
 
 Since, $h$ has infinite order in $\pi(M)$, $3x+\mu_{1}+\mu_{2}=0.$ In that case, $g$ is a generalised $3$-torsion element in $\pi_1(M)$.
 \end{proof}

\begin{rk}
    We can find some sufficient conditions for the existence of generalised $n$-torsion elements in a Seifert-fibered space following similar arguments as above. In particular, if there exist elements $c_i$ for $i=1,2$ such that $c_i^{np_i}=h^{\mu_i}$ in the fundamental group, then any conjugate of $c_1^{p_1}c_2^{p_2}h^x$ is a generalised $n$-torsion element if $nx + \mu_1 + \mu_2 =0$.
\end{rk}
 
 In the particular case of the braid group $B_{3}$, we now have the following corollary. 
 \begin{cor}
 	Consider $M=S^3-K$ where $K$ is the trefoil knot, and hence $\pi_1(M)=B_3=\langle c_1,c_2,h|c_1^2=c_2^3=h, c_ih=hc_i\rangle$. Then the elements which are conjugate to $e_ie_j^2h^{-1}$, where  $e_i$'s are distinct and conjugate to $c_2$, are generalised $3$-torsion elements.  
 \end{cor}

\begin{proof}
Let us choose two $3$-ordered elements $[ e_{1}],[e_{2}]$ in the quotient group $B_{3}/\langle h\rangle$, conjugate to $[c_2].$
 	So, in $B_{3}$,
 	$e_{1}^{3}=h=e_{2}^{3}.$ In $\rm PSL(2,\z)$ we have noticed that the product of two $3$-ordered elements are generalised $3$-torsion elements. 
 	
 	So, we have two choices of generalised 3-torsion elements. For the first case,
 	we take one of the lifts of $e_{1}e_{2}$, say $g=e_{1}e_{2}h^{x}$, where $x\in \z$. Let g satisfy the condition \[ gh_{1}gh_{1}^{-1}kgk^{-1}=Id,\] and let us choose $h_{1}=e_{2}^{-1}$ and $k=e_{2}$. Then,
 	\[ e_{1}e_{2}h^{x}e_{2}^{-1}e_{1}e_{2}h^{x}e_{2}e_{2}e_{1}e_{2}h^{x}e_{2}^{-1}=Id\]
 	\[\imp e_{1}h^{x}e_{1}h^{x}e_{2}^{3}e_{1}h^{x}=Id\]
 	\[\imp h^{3x+2}=Id.\]
 	This case is not possible, since $h$ has infinite order and there is no integer $x$ such that $3x+2=0.$
 	For the next case we may consider $g$ as a lift of $e_{1}e_{2}^2$, i.e., $g=e_{1}e_{2}^{2}h^{x}.$ Now, if we continue the same process as previous case, we obtain
 	$h^{3x+3}=Id.$
 	Therefore, for $x=-1$, $g$ will be a generalised $3$-torsion element.
 \end{proof}

\end{document}